\documentclass[12pt]{amsart}
\usepackage[cp1251]{inputenc}
\usepackage[T2A]{fontenc}
\usepackage[english]{babel}
\usepackage{amsmath,amsfonts,amssymb}
\usepackage{geometry}
\usepackage{amsmath, amsthm, amscd, amsfonts, amssymb, graphicx, color}
\usepackage[bookmarksnumbered, colorlinks, plainpages]{hyperref}

\newtheorem{theorem}{Theorem}

\newtheorem{proposition}{Proposition}

\theoremstyle{remark}

\theoremstyle{definition}

\title[Elliptic problems with unknowns on the boundary]{Elliptic problems\\ with unknowns on the boundary\\and irregular boundary data}

\author[I. Chepurukhina]{Iryna Chepurukhina}

\address{Institute of Mathematics of the National Academy of Sciences of Ukraine, Tereshchen\-kiv\-s'ka 3, Kyiv 01024, Ukraine}

\email{Chepuruhina@gmail.com}

\author[A. Murach]{Aleksandr Murach}

\address{Institute of Mathematics of the National Academy of Sciences of Ukraine, Tereshchen\-kiv\-s'ka 3, Kyiv 01024, Ukraine}

\email{murach@imath.kiev.ua}



\begin{document}

\maketitle

\begin{abstract}
We consider an elliptic problem with unknowns on the boundary of the domain of the elliptic equation and suppose that the right-hand side of this equation is square integrable and that the boundary data are arbitrary (specifically, irregular) distributions. We investigate local (up to the boundary) properties of generalized solutions to the problem in Hilbert distribution spaces that belong to the refined Sobolev scale. These spaces are parametrized with a real number and a function that varies slowly at infinity. The function parameter refines the number order of the space.
We prove theorems on local regularity and a local \textit{a priori} estimate of generalized solutions to the problem under investigation. These theorems are new for Sobolev spaces as well.
\end{abstract}

\section{Introduction}\label{sec1}

In the theory of elliptic boundary-value problems, of special interest is the case where boundary data are irregular distributions (so called rough data); see monographs \cite{BehrndtHassideSnoo20, Berezansky68, Hermander63, KozlovMazyaRossmann97, LionsMagenes72, MikhailetsMurach14, Roitberg96, Roitberg99} and references therein. The investigation of elliptic problems in this case is more complicated as compared with regular enough boundary data. This is stipulated by the fact that the trace theorems for Sobolev or other classical distribution spaces (see, e.g., \cite[Section~4.7]{Triebel95}) cease to be valid for irregular boundary data. There are some approaches to overcome this serious obstacle. One of them is to include norms of traces in Sobolev norms of solutions to an elliptic problem. This approach was elaborated by Roitberg \cite{Roitberg64, Roitberg65, Roitberg96}; it deals with solutions that are not distributions, generally speaking. Another way is to include a relevant norm of the right-hand side of the elliptic equation in the norms of solutions. This approach is due to Lions and Magenes \cite{LionsMagenes62V, LionsMagenes63VI, LionsMagenes72}; it remains in the framework of distribution spaces but is applicable to narrower classes of data of the elliptic equation as compared with Roitberg's approach. Investigating elliptic problems in a half-space, H\"ormander \cite[Section~10.4]{Hermander63} used anisotropic Sobolev spaces of high enough regularity only along the normal to the boundary. These approaches yield different solvability theorems for elliptic problems whose boundary data are arbitrary distributions.

The number of boundary conditions increases if an elliptic problem contains unknown distributions on the boundary. Such elliptic problems were first considered by Lawruk \cite{Lawruk63a, Lawruk63b}. They form a part of Boutet de Monvel's algebra \cite{BoutetdeMonvel71} and arise in various applications, specifically in hydrodynamics and the theory of elasticity \cite{AslanyanVassilievLidskii81, Ciarlet90}, and are also used in the theory of free boundary problems \cite{NazarovPileckas93}. A solvability theory for elliptic problems with unknowns on the boundary is given in monographs \cite[Part~1]{KozlovMazyaRossmann97} and \cite[Chapter~2]{Roitberg99} in the framework of Roitberg's approach. Kozhevnikov \cite{Kozhevnikov01} extended this approach to pseudodifferential elliptic problems that form the Boutet de Monvel algebra.

In this connection, it is interesting to investigate these problems in the spirit of the approach by Lions and Magenes. Thus, recently we proved a corresponding solvability theorem for elliptic problems with unknowns on the boundary \cite[Theorem~1]{MurachChepurukhina15UMJ5}. We assumed that the right-hand side of the elliptic equation is square integrable and considered boundary data in Hilbert distribution spaces of lower orders (including negative ones). These spaces belong to the refined Sobolev scale \cite[Section~2.1]{MikhailetsMurach14}. The purpose of the present paper is to supplement this result with theorems on local (up to the boundary) regularity and a local \textit{a priori} estimate of generalized solutions to the problem. In contrast to the corresponding global properties of the solutions \cite{MurachChepurukhina15UMJ5}, these theorems do not follow directly from the solvability theorem. Specifically, this is caused by the fact that the space of solutions (to the elliptic equation) used in \cite{MurachChepurukhina15UMJ5} is not closed with respect to the multiplication of distributions by smooth cut-off functions.

\section{Statement of the problem}\label{sec2}

Let $\Omega $ be a bounded Euclidean domain of dimension $n\geq2$ with an infinitely smooth boundary $\Gamma$. We arbitrarily choose integers $q\geq1$, $\varkappa\geq1$, $m_{1},\ldots,m_{q+\varkappa}\leq2q-1$ and $r_{1},\ldots,r_{\varkappa}$. We consider the following boundary-value problem in $\Omega$:
\begin{gather}\label{10f1}
Au=f\quad\mbox{in}\quad\Omega, \\
B_{j}u+\sum_{k=1}^{\varkappa}C_{j,k}v_{k}=g_{j}\quad \mbox{on} \quad \Gamma,\quad j=1, ...,q+\varkappa.\label{10f2}
\end{gather}
Here, $A:=A(x,D)$ is a linear partial differential operator (PDO) on $\overline{\Omega}:=\Omega\cup\Gamma$ of the even order $2q$; each $B_{j}:=B_{j}(x,D)$ is a linear boundary PDO on $\Gamma$ whose order
$\mathrm{ord}\,B_{j}\leq m_{j}$, and each $C_{j,k}:=C_{j,k}(x,D_{\tau})$ is a linear tangent PDO on~$\Gamma$ whose order $\mathrm{ord}\,C_{j,k}\leq m_{j}+r_{k}$. (As usual, PDOs of negative order are defined to be zero operators.) We assume that all coefficients of the indicated PDOs are infinitely smooth functions on $\overline{\Omega}$ or $\Gamma$ respectively. The distribution $u$ on $\Omega$ and the distributions $v_{1},\ldots,v_{\varkappa}$ on $\Gamma$ are unknown in this problem. In the paper, all functions and distributions are supposed to be complex-valued; we therefore use complex distribution/function spaces.

We assume that $m\geq-r_{k}$ for each $k\in\{1,\ldots,\varkappa\}$.
This assumption is natural; indeed, if $m+r_{k}<0$ for some $k$, then $C_{1,k}=\cdots=C_{q+\varkappa,k}=0$, i.e. the boundary conditions \eqref{10f2} will not contain the unknown $v_{k}$.

We suppose that the boundary problem \eqref{10f1}, \eqref{10f2} is elliptic in $\Omega$ as a problem with additional unknown distributions   $v_{1},\ldots,v_{\varkappa}$ on $\Gamma$. This means that the PDO $A$ is properly elliptic on $\overline{\Omega}$, and the system of boundary conditions \eqref{10f2} covers  $A$ on $\Gamma$ (see, e.g., \cite[Section~3.1.3]{KozlovMazyaRossmann97}). We recall the corresponding definitions.

Let $A^{\circ}(x,\xi)$, $B_{j}^{\circ}(x,\xi)$, and $C_{j,k}^{\circ}(x,\tau)$ denote the principal symbols of the PDOs $A(x,D)$, $B_{j}(x,D)$, and $C_{j,k}(x,D_{\tau})$ respectively, the last two PDOs being considered as that of the formal orders $m_{j}$ and $m_{j}+r_{k}$ respectively. Thus, $A^{\circ}(x,\xi)$ and $B_{j}^{\circ}(x,\xi)$ are homogeneous polynomials in $\xi\in\mathbb{C}^{n}$ of order $2q$ and $m_{j}$ respectively. Besides, $C_{j,k}^{\circ}(x,\tau)$ is a homogeneous polynomial of order $m_{j}+r_{k}$ in $\tau$, where $\tau$ is a tangent vector to the boundary $\Gamma$ at the point~$x$. Defining the principal symbols, we consider the principal parts of the PDOs as polynomials in $D_l:=i\partial/\partial x_l$, where $l=1,\ldots,n$, and then replace each differential operator $D_l$ with the $l$-th component $\xi_l$ of the vector $\xi$.

The boundary-value problem \eqref{10f1}, \eqref{10f2} is called elliptic in $\Omega$ if it satisfies the following two conditions:
\begin{itemize}
\item [(i)] The PDO $A(x,D)$ is properly elliptic at every point $x\in\overline{\Omega}$; i.e., for arbitrary linear independent vectors $\xi',\xi''\in\mathbb{R}^{n}$, the polynomial $A^{\circ}(x,\xi'+\zeta\xi'')$ in $\zeta\in\mathbb{C}$ has $q$ roots with positive imaginary part and $q$ roots with negative imaginary part (of course, these roots are calculated with regard for their multiplicity).
\item [(ii)] The boundary conditions \eqref{10f2} cover $A(x,D)$ at every point $x\in\Gamma$. This means that, for an arbitrary tangent vector $\tau\neq0$ to $\Gamma$ at $x$, the boundary-value problem
    \begin{gather*}
    A^{\circ}(x,\tau+\nu(x)D_{t})\theta(t)=0\quad\mbox{for}\;\;t>0,\\
    B_{j}^{\circ}(x,\tau+\nu(x)D_{t})\theta(t)\big|_{t=0}+
    \sum_{k=1}^{\varkappa}C_{j,k}^{\circ}(x,\tau)\lambda_{k}=0,
    \quad j=1,...,q+\varkappa,\\
    \theta(t)\to0\quad\mbox{as}\quad t\rightarrow\infty
    \end{gather*}
    has only the trivial (zero) solution. Here, the function  $\theta\in C^{\infty}([0,\infty))$ and the numbers $\lambda_{1},\ldots,\lambda_{\varkappa}\in\mathbb{C}$ are unknown, whereas $\nu(x)$ is the unit inward normal vector to $\Gamma$ at $x$. Besides, $A^{\circ}(x,\tau+\nu(x)D_{t})$ and $B_{j}^{\circ}(x,\tau+\nu(x)D_{t})$ are differential operators with respect to $D_{t}:=i\partial/\partial t$. We obtain them putting $\zeta:=D_{t}$ in the polynomials $A^{\circ}(x,\tau+\zeta\nu(x))$ and $B_{j}^{\circ}(x,\tau+\zeta\nu(x))$ in $\zeta$, respectively.
\end{itemize}

Some examples of elliptic problems of the form \eqref{10f1}, \eqref{10f2} are given in \cite[Subsection~3.1.5]{KozlovMazyaRossmann97}.

\section{A refined Sobolev scale}\label{sec3}

This scale consists of the Hilbert generalized Sobolev spaces $H^{s,\varphi}$ whose order of regularity is given by a number  $s\in\mathbb{R}$ and function $\varphi\in\mathcal{M}$. Here, $\mathcal{M}$ denotes the set of all Borel measurable functions $\varphi:[1,+\infty)\rightarrow(0,+\infty)$ such that both functions $\varphi$ and $1/\varphi$ are bounded on each compact subset of $[1,+\infty)$ and that $\varphi$ varies slowly at  infinity in the sense of Karamata \cite{Karamata30a}, i.e. $\varphi(\lambda t)/\varphi(t)\rightarrow 1$ as $t\rightarrow\infty$ for every $\lambda>0$.

Slowly varying functions are well studied and have various important applications \cite{BinghamGoldieTeugels89}. A standard example  of such functions is
$$
\varphi(t):=(\log t)^{r_{1}}(\log\log
t)^{r_{2}}\ldots(\underbrace{\log\ldots\log}_{k\;\mbox{\small times}}
t)^{r_{k}}\quad\mbox{of}\quad t\gg1,
$$
where $k\in\mathbb{N}$ and $r_{1},\ldots,r_{k}\in\mathbb{R}$.

Let $s\in\mathbb{R}$ and $\varphi\in\mathcal {M}$. By definition, the linear space $H^{s,\varphi}(\mathbb{R}^{n})$, with $n\geq1$, consists of all distributions $w\in\mathcal{S}'(\mathbb{R}^{n})$ that their Fourier transform $\widehat{w}$ is locally Lebesgue integrable over $\mathbb{R}^{n}$ and satisfies the condition
$$
\|w\|_{s,\varphi;\mathbb{R}^{n}}^{2}:=
\int\limits_{\mathbb{R}^{n}}\langle\xi\rangle^{2s}
\varphi^{2}(\langle\xi\rangle)\,
|\widehat{w}(\xi)|^{2}\,d\xi<\infty.
$$
Here, $\mathcal{S}'(\mathbb{R}^{n})$ is the linear topological space of all tempered distributions on $\mathbb{R}^{n}$, and
$\langle\xi\rangle:=(1+|\xi|^{2})^{1/2}$. By definition, $\|\cdot\|_{s,\varphi;\mathbb{R}^{n}}$ is the norm in $H^{s,\varphi}(\mathbb{R}^{n})$.

The space $H^{s,\varphi}(\mathbb{R}^{n})$ is a special isotropic Hilbert case of the spaces introduced and investigated by H\"ormander \cite[Section~2.2]{Hermander63} (see also his monograph \cite[Section~10.1]{Hermander83}) and by Volevich and Paneah \cite[\S~2]{VolevichPaneah65}. If $\varphi(\cdot)\equiv1$, the space $H^{s,\varphi}(\mathbb {R}^{n})$ becomes the inner product Sobolev space $H^{s}(\mathbb {R}^{n})$ of order $s\in\mathbb{R}$. Generally, we have the dense continuous embeddings
\begin{equation}\label{10f3}
H^{s+\varepsilon}(\mathbb{R}^{n})\hookrightarrow H^{s,\varphi}(\mathbb{R}^{n})\hookrightarrow H^{s-\varepsilon}(\mathbb{R}^{n})
\quad\mbox{whenever}\quad\varepsilon>0.
\end{equation}
They show that the function parameter $\varphi$ refines the main regularity characterized by the number $s$. Therefore, the class of spaces $H^{s,\varphi}(\mathbb{R}^{n})$, where $s\in\mathbb{R}$ and $\varphi\in\mathcal{M}$, was called the refined Sobolev scale over $\mathbb{R}^{n}$ \cite[Section~1.3.3]{MikhailetsMurach14}. This class was selected and investigated in \cite{MikhailetsMurach05UMJ5, MikhailetsMurach06UMJ3} (compare, e.g., with \cite[Chapter~III]{Triebel01} and \cite{HaroskeMoura04}, where similar classes of Banach and more general spaces of distributions were studied).

To investigate the boundary-value problem \eqref{10f1}, \eqref{10f2}, we
need versions of the space $H^{s,\varphi}(\mathbb{R}^{n})$ for $\Omega$ and $\Gamma$; they are considered in \cite[Sections 2.1 and 3.2.1]{MikhailetsMurach14}.

By definition, the linear space $H^{s,\varphi}(\Omega)$ consists of
the restrictions of all distributions $w\in H^{s,\varphi}(\mathbb{R}^{n})$ to $\Omega$. It is endowed with the norm
$$
\|u\|_{s,\varphi;\Omega}:=
\inf\,\bigl\{\,\|w\|_{s,\varphi;\mathbb{R}^{n}}:
w\in H^{s,\varphi}(\mathbb{R}^{n}),\;\,
w=u\;\,\mbox{in}\;\,\Omega\,\bigr\},
$$
where $u\in H^{s,\varphi}(\Omega)$. The space $H^{s,\varphi}(\Omega)$ is Hilbert and separable with respect to this norm, with $C^{\infty}(\overline{\Omega})$ being a dense subset of  this space.

Briefly saying, the space $H^{s,\varphi}(\Gamma)$ consists of all distributions on $\Gamma$ that are reduced to distributions from $H^{s,\varphi}(\mathbb{R}^{n-1})$ in local coordinates on $\Gamma$. Let us give a detailed definition. We arbitrarily choose a finite collection of infinitely smooth local charts $\pi_j:\mathbb{R}^{n-1}\leftrightarrow\Gamma_{j}$, with $j=1,\ldots,\lambda$, that the open sets $\Gamma_{1},\ldots,\Gamma_{\lambda}$ form a covering of $\Gamma$. We also arbitrarily choose functions $\chi_j\in C^{\infty}(\Gamma)$, with $j=1,\ldots,\lambda$, that form a partition of unity on $\Gamma$ subject to $\mathrm{supp}\,\chi_j\subset\Gamma_j$. By definition, the linear space $H^{s,\varphi}(\Gamma)$ consists of all distributions $h\in\mathcal{D}'(\Gamma)$ such that $(\chi_{j}h)\circ\pi_{j}\in H^{s,\varphi}(\mathbb{R}^{n-1})$ for each $j\in\{1,\ldots,\lambda\}$.
Here, $\mathcal{D}'(\Gamma)$ is the linear topological space of all distributions on $\Gamma$, and $(\chi_{j}h)\circ\pi_{j}$ stands for  the representation of the distribution $\chi_{j}h$ in the local chart $\pi_{j}$. The norm in $H^{s,\varphi}(\Gamma)$ is defined by the formula
$$
\|h\|_{s,\varphi;\Gamma}:=
\biggl(\,\sum_{j=1}^{\lambda}\,\|(\chi_{j}h)\circ\pi_{j}\|_
{s,\varphi;\mathbb{R}^{n-1}}^{2}\biggr)^{1/2}.
$$
The space $H^{s,\varphi}(\Gamma)$ is Hilbert and separable. It does not depend (up to equivalence of norms) on the indicated choice of local charts and partition of unity on $\Gamma$ \cite[Theorem~2.3]{MikhailetsMurach14}. The set $C^{\infty}(\Gamma)$ is dense in $H^{s,\varphi}(\Gamma)$.

The spaces $H^{s,\varphi}(\Omega)$ and $H^{s,\varphi}(\Gamma)$, where $s\in\mathbb{R}$ and $\varphi\in\mathcal{M}$, form the refined Sobolev scales over $\Omega$ and $\Gamma$. If $\varphi(\cdot)\equiv1$, these spaces become the inner product Sobolev spaces $H^{s}(\Omega)$ and $H^{s}(\Gamma)$, the norms in them being denoted by $\|\cdot\|_{s;\Omega}$ and $\|\cdot\|_{s;\Gamma}$, resp. Generally, the dense compact embeddings \eqref{10f3} hold true provided that we replace $\mathbb{R}^{n}$ with $\Omega$ or $\Gamma$.

The refined Sobolev scale over $G\in\{\mathbb{R}^{n},\Omega,\Gamma\}$ possesses the following important interpolation property: every space $H^{s,\varphi}(G)$ is obtained by quadratic interpolation (with an appropriate function parameter) between the Sobolev spaces $H^{s-\varepsilon}(G)$ and $H^{s+\delta}(G)$ where $\varepsilon,\delta>0$ (see \cite[Theorems 1.14, 2.2, and 3.2]{MikhailetsMurach14}). This property play a key role in applications of these scales to elliptic operators and elliptic problems (see \cite{MikhailetsMurach12BJMA2, MikhailetsMurach14} and references therein).

In what follows we will consider various Hilbert spaces induced by the spaces $H^{s,\varphi}(G)$ and related to the problem \eqref{10f1}, \eqref{10f2}. If $\varphi(\cdot)\equiv1$, we will omit the index $\varphi$ in the designations of these spaces and norms in them.

\section{Main results}\label{sec4}

Consider the linear mapping
\begin{equation}\label{mapping}
\Lambda:(u,v_{1},...,v_{\varkappa})\mapsto(f,g_{1},...,g_{q+\varkappa}),
\;\;\mbox{where}\;\;u\in C^{\infty}(\overline{\Omega}),\;\;
v_{1},\ldots,v_{\varkappa}\in C^{\infty}(\Gamma)
\end{equation}
and where the functions $f$ and $g_{1}$,..., $g_{q+\varkappa}$ are defined by formulas \eqref{10f1} and \eqref{10f2}. Introduce the Hilbert spaces
\begin{equation*}
\mathcal{D}^{s,\varphi}(\Omega,\Gamma):=H^{s,\varphi}(\Omega)\oplus
\bigoplus_{k=1}^{\varkappa}H^{s+r_{k}-1/2,\varphi}(\Gamma)
\end{equation*}
and
\begin{equation*}
\mathcal{E}_{s,\varphi}(\Omega,\Gamma):=H^{s-2q,\varphi}(\Omega)\oplus
\bigoplus_{j=1}^{q+\varkappa}H^{s-m_{j}-1/2,\varphi}(\Gamma)
\end{equation*}
for arbitrary $s\in\mathbb{R}$ and $\varphi\in\mathcal{M}$.

According to \cite[Theorem~1]{ChepurukhinaMurach15MFAT1}, this mapping extends uniquely (by continuity) to a Fredholm bounded operator
\begin{equation}\label{Fredholm-positive-scale}
\Lambda:\mathcal{D}^{s,\varphi}(\Omega,\Gamma)\to
\mathcal{E}_{s,\varphi}(\Omega,\Gamma)
\end{equation}
for all $s>2q-1/2$ and $\varphi\in\mathcal{M}$. The finite-dimensional kernel of the operator \eqref{Fredholm-positive-scale} lies in
$$
\mathcal{D}^{\infty}(\overline{\Omega},\Gamma):=
C^{\infty}(\overline{\Omega})\times(C^{\infty}(\Gamma))^{\varkappa}
$$
and together with the finite index of \eqref{Fredholm-positive-scale} does not depend on $s$ and~$\varphi$. Let $N$ denote the kernel, and let $\vartheta$ stand for the index.

This result cannot be spread to all real $s$ without changes in its formulation. This follows from the known fact that the trace operator $u\mapsto u\!\upharpoonright\!\Gamma$, where $u\in C^{\infty}(\overline{\Omega})$, cannot be extended to a continuous mapping from $H^{s}(\Omega)$ to $\mathcal{D}'(\Gamma)$ if $s\leq1/2$.

In the $s\leq2q-1/2$ case, the boundary data $g_{j}\in H^{s-m_{j}-1/2,\varphi}(\Gamma)$ may be irregular distributions
(so called, rough data). Examining this case, we assume that $f\in L_{2}(\Omega)$, which allows us to use an $s<2q$ version \cite[Theorem~1]{MurachChepurukhina15UMJ5} of the above result. This version involves the linear space
$$
H^{s,\varphi}_{A}(\Omega):=\bigl\{u\in H^{s,\varphi}(\Omega):
Au\in L_{2}(\Omega)\bigr\}
$$
endowed with the graph norm
\begin{equation*}
\|u\|_{s,\varphi;\Omega,A}:=
\bigl(\|u\|^{2}_{s,\varphi;\Omega}+\|Au\|^{2}_{\Omega}\bigr)^{1/2}.
\end{equation*}
Here, $s<2q$; $\varphi\in\mathcal{M}$; $\|\cdot\|_{\Omega}$ is the norm in the Hilbert space $L_{2}(\Omega)$ of square integrable functions over $\Omega$, and $Au$ is understood in the sense of the theory of distributions on $\Omega$. The space $H^{s,\varphi}_{A}(\Omega)$ is Hilbert, and $C^{\infty}(\overline{\Omega})$ is dense in this space   \cite[Section~4]{MurachChepurukhina15UMJ5}. Note that $H^{s,\varphi}_{A}(\Omega)$ depends essentially on $A$ (even when all coefficients of $A$ are constant), which was shown by H\"ormander \cite[Theorem~3.1]{Hermander55}
in the case where $s=0$ and $\varphi(\cdot)\equiv1$. Consider the Hilbert spaces
\begin{equation*}
\mathcal{D}^{s,\varphi}_{A}(\Omega,\Gamma):=H^{s,\varphi}_{A}(\Omega)\oplus
\bigoplus_{k=1}^{\varkappa}H^{s+r_{k}-1/2,\varphi}(\Gamma)
\end{equation*}
and
\begin{equation*}
\mathcal{E}^{0,s,\varphi}(\Omega,\Gamma):=L_{2}(\Omega)\oplus
\bigoplus_{j=1}^{q+\varkappa}H^{s-m_{j}-1/2,\varphi}(\Gamma).
\end{equation*}

\begin{proposition}\label{10pr1}
The mapping \eqref{mapping} extends uniquely (by continuity) to a bounded operator
\begin{equation}\label{10f5}
\Lambda:\mathcal{D}^{s,\varphi}_{A}(\Omega,\Gamma)\rightarrow
\mathcal{E}^{0,s,\varphi}(\Omega,\Gamma)
\end{equation}
for arbitrary $s<2q$ and $\varphi\in\mathcal{M}$. This operator is Fredholm. Its kernel coincides with $N$, and its index equals $\vartheta$.
\end{proposition}

This result was proved in \cite[Theorem~1]{MurachChepurukhina15UMJ5}.
We will supplement it with theorems on local (up to the boundary $\Gamma$) regularity and a local \textit{a priori} estimate of the generalized solutions to the elliptic problem \eqref{10f1},~\eqref{10f2}. Beforehand, using Proposition~\ref{10pr1}, we give a definition of such a solution.

Put
$$
\mathcal{S}'_{A}(\Omega):=\{u\in\mathcal{S}'(\Omega):
Au\in L_{2}(\Omega)\},
$$
where, as usual, $\mathcal{S}'(\Omega)$ is the space of the restrictions of all distributions $w\in\mathcal{S}'(\mathbb{R}^{n})$ to~$\Omega$.
Since $\Omega$ is bounded, the space $\mathcal{S}'_{A}(\Omega)$ is the union of all $H^{s,\varphi}_{A}(\Omega)$ such that $s<2q$ and $\varphi\in\mathcal{M}$.

Assume that
\begin{equation}\label{10f6}
(u,v):=(u,v_{1},\ldots,v_{\varkappa})\in
\mathcal{S}'_{A}(\Omega)\times(\mathcal{D}'(\Gamma))^{\varkappa}
\end{equation}
and
\begin{equation*}
(f,g):=(f,g_{1},\ldots,g_{q+\varkappa})\in
L_{2}(\Omega)\times(\mathcal{D}'(\Gamma))^{q+\varkappa}.
\end{equation*}
The vector \eqref{10f6} is called a generalized (strong) solution to the boundary-value problem \eqref{10f1}, \eqref{10f2} if $\Lambda(u,v)=(f,g)$ for some operator \eqref{10f5} from Proposition~\ref{10pr1}. This definition is reasonable because $(u,v)\in\mathcal{D}^{s,\varphi}_{A}(\Omega,\Gamma)$ for sufficiently small $s<2q$ and every $\varphi\in\mathcal{M}$ and because the image $\Lambda(u,v)$ does not depend on these $s$ and $\varphi$.

Now we introduce local versions of the spaces $H^{l,\varphi}(\Omega)$ and $H^{l,\varphi}(\Gamma)$, where $l\in\mathbb{R}$ and $\varphi\in\mathcal{M}$. We need these versions to formulate a theorem on local regularity of a generalized solution to the problem under investigation. Let $U$ be an open subset of $\mathbb{R}^{n}$ such that $\Omega_{0}:=\Omega\cap U\neq\emptyset$ and $\Gamma_{0}:=\Gamma\cap U\neq\emptyset$. We let $H^{l,\varphi}_{\mathrm{loc}}(\Omega_{0},\Gamma_{0})$ denote the linear space of all distributions $u\in\mathcal{S}'(\Omega)$ such that $\chi u\in H^{l,\varphi}(\Omega)$ for every function $\chi\in C^{\infty}(\overline{\Omega})$ satisfying $\mathrm{supp}\,\chi\subset\Omega_0\cup\Gamma_{0}$. Analogously,
$H^{l,\varphi}_{\mathrm{loc}}(\Gamma_{0})$ denotes the linear space of all distributions $h\in\mathcal{D}'(\Gamma)$ such that $\chi h\in H^{l,\varphi}(\Gamma)$ for every function $\chi\in C^{\infty}(\Gamma)$ satisfying $\mathrm{supp}\,\chi\subset\Gamma_{0}$.

\begin{theorem}\label{10th1}
Let $s<2q$ and $\varphi\in\mathcal{M}$. Assume that a vector \eqref{10f6} is a generalized solution to the elliptic problem \eqref{10f1}, \eqref{10f2} whose right-hand sides satisfy the conditions $f\in L_2(\Omega)$ and $g_j\in H^{s-m_j-1/2,\varphi}_{\mathrm{loc}}(\Gamma_0)$ for each $j\in\{1,\ldots,q+\varkappa\}$. Then $u\in H^{s,\varphi}_{\mathrm{loc}}(\Omega_{0},\Gamma_{0})$ and  $v_k\in H^{s+r_k-1/2,\varphi}_{\mathrm{loc}}(\Gamma_0)$ for each $k\in\{1,\ldots,\varkappa\}$.
\end{theorem}

Note that the definition of $H^{l,\varphi}_{\mathrm{loc}}(\Omega_{0},\Gamma_{0})$ makes sense in the $\Gamma_{0}=\emptyset$ case. It follows from condition \eqref{10f6} and the ellipticity of the PDO $A$ that $u\in H^{2q}_{\mathrm{loc}}(\Omega_{0},\emptyset)$ (see, e.g., \cite[Theorem 7.4.1]{Hermander63}).

Now we formulate a theorem on a local \textit{a priori} estimate of the generalized solution to the problem under investigation. Let $\|\cdot\|'_{s,\varphi}$ denote the norm in the Hilbert space $\mathcal{D}^{s,\varphi}(\Omega,\Gamma)$, and let $\|\cdot\|''_{0,s,\varphi}$ stand for the norm in the Hilbert space $\mathcal{E}^{0,s,\varphi}(\Omega,\Gamma)$.

\begin{theorem}\label{10th2}
Let $s<2q$ and $\varphi\in\mathcal{M}$. Assume that a vector \eqref{10f6} satisfies the hypotheses of Theorem~$\ref{10th1}$. We arbitrarily choose a number $\lambda>0$ and functions $\chi,\eta\in C^{\infty}(\overline{\Omega})$ such that $\mathrm{supp}\,\chi\subset\mathrm{supp}\,\eta\subset
\Omega_{0}\cup\Gamma_{0}$ and that $\eta=1$ in a neighbourhood of $\mathrm{supp}\,\chi$. Then
\begin{equation}\label{10f9}
\|\chi(u,v)\|'_{s,\varphi}\leq c\,\bigl(\|\eta(f,g)\|''_{0,s,\varphi}+
\|\eta(u,v)\|'_{s-\lambda,\varphi}\bigr)
\end{equation}
for a certain number $c>0$ that does not depend on $(u,v)$ and $(f,g)$.
\end{theorem}

Here, of course, $\chi(u,v)$ means $(\chi u,(\chi\!\upharpoonright\!\Gamma)v_{1},\ldots,
(\chi\!\upharpoonright\!\Gamma)v_{\varkappa})$, and $\eta(f,g)$ is analogously interpreted. These theorems are new in the Sobolev case of $\varphi(\cdot)\equiv1$ even where $s$ is an integer. They consist the local (up to the boundary) lifting property of the generalized solution
$(u,v)$.

\section{Proofs}\label{sec5}

If $\Omega_{0}=\Omega$ and $\Gamma_{0}=\Gamma$
and if $\chi(\cdot)\equiv\eta(\cdot)\equiv1$, Theorems \ref{10th1} and \ref{10th2} deal with global properties of the generalized solution $(u,v)$, i.e. with its properties in the whole domain $\Omega$ up to the boundary $\Gamma$. In this specific case, the theorems follows easily from Proposition~\ref{10pr1} and are given in \cite[Theorems 3 and~2]{MurachChepurukhina15UMJ5}. In the general case, Theorems \ref{10th1} and \ref{10th2} cannot be deduced from the global case in a routine manner used in \cite[Section~3.2.3]{KozlovMazyaRossmann97} and \cite[Section~2.4.4]{Roitberg99} for elliptic problems with unknowns on the boundary. This is caused by the following two circumstances: the space $H^{s,\varphi}_{A}(\Omega)$ is not closed with respect to the multiplication by functions from $C^{\infty}(\overline{\Omega})$, and the right-hand side of the inequality \eqref{10f9} contains the norm $\|\eta f\|_{\Omega}$ instead of $\|\eta f\|_{s-2q,\varphi}$. We therefore cannot take $\chi u$ instead of $u$ in the global versions of these theorems to treat the general case. Besides, if we interchange the PDO $A$ and the operator of the multiplication by $\chi$ according to the routine, we get
$$
\|A(\chi u)\|_{\Omega}\leq\|\chi Au\|_{\Omega}+\|A'(\eta u)\|_{\Omega}\leq\|\chi f\|_{\Omega}+\|\eta u\|_{2q-1;\Omega}
$$
whenever $u\in C^{\infty}(\overline{\Omega})$, which yields a trivial estimate (for $\chi u$) instead of \eqref{10f9} provided that $\nobreak{s<2q-1}$ (here, the PDO $A'$ is the commutator of these operators).

To prove Theorems \ref{10th1} and \ref{10th2}, we develop methods worked out in \cite[Section~5]{AnopDenkMurach20arxiv} and \cite[Section~6]{AnopKasirenkoMurach18UMJ3} for elliptic problems without unknowns on the boundary. These methods use property of elliptic problems in Sobolev spaces modified by Roitberg \cite{Roitberg64, Roitberg65} (see also his monograph \cite[Section~2]{Roitberg96}). For our purposes, we need the similar modification $H^{s,\varphi,(2q)}(\Omega)$ of the space $H^{s,\varphi}(\Omega)$. This modification was introduced and investigated in \cite{MikhailetsMurach08UMJ4} (see also the book \cite[Section~4.2.2]{MikhailetsMurach14}). In the Sobolev case of $\varphi(\cdot)\equiv1$, the space $H^{s,\varphi,(2q)}(\Omega)$ was introduced by Roitberg in~\cite{Roitberg64}.

Let $\varphi\in\mathcal{M}$. If $s\notin\{1/2,3/2,\ldots,2q-1/2\}$, then $H^{s,\varphi,(2q)}(\Omega)$ is defined to be the completion of $C^{\infty}(\overline{\Omega})$ with respect to the Hilbert norm
$$
\|u\|_{s,\varphi,(2q);\Omega}:=
\biggl(\|u\|_{s,\varphi,(0);\Omega}^{2}+
\sum_{j=1}^{k}\;\|(\partial_{\nu}^{j-1}u)\!\upharpoonright\!\Gamma\|
_{s-j+1/2,\varphi;\Gamma}^{2}\biggr)^{1/2}.
$$
Here, $\partial_{\nu}$ is the operator of the differentiation with respect to the inward normal to $\Gamma$, and $\|\cdot\|_{s,\varphi,(0);\Omega}$ is the norm in the Hilbert space
$$
H^{s,\varphi,(0)}(\Omega):=
\left\{\begin{array}{ll}
H^{s,\varphi}(\Omega)&\hbox{if}\;\;s\geq0\\
(H^{-s,1/\varphi}(\Omega))'&\hbox{if}\;\;s<0,
\end{array}\right.
$$
the dual space being considered with respect to the inner product in $L_{2}(\Omega)$. (Note that $\varphi\in\mathcal{M}\Leftrightarrow 1/\varphi\in\mathcal{M}$.) Thus, if $s<0$, the space $H^{s,\varphi,(0)}(\Omega)$ is the completion of $L_{2}(\Omega)$ with respect to the norm
$$
\|u\|_{s,\varphi,(0);\Omega}:=\sup\biggl\{\frac{|(u,w)_{\Omega}|}
{\;\quad\|w\|_{-s,1/\varphi;\Omega}}:
w\in H^{-s,1/\varphi}(\Omega),\,w\neq0\biggr\},
$$
where $(\cdot,\cdot)_{\Omega}$ stands for the inner product in $L_{2}(\Omega)$. Then  $\|u\|_{s,\varphi,(0);\Omega}=\|\mathcal{O}u\|_{s,\varphi;\mathbb{R}^{n}}$
for every $u\in C^{\infty}(\overline{\Omega})$, with $\mathcal{O}u:=u$ on $\overline{\Omega}$ and $\mathcal{O}u:=0$ on $\mathbb{R}^{n}\setminus\overline{\Omega}$. If $s\in\{1/2,3/2,\ldots,2q-1/2\}$, the Hilbert space $H^{s,\varphi,(2q)}(\Omega)$ is, by definition, the result of the quadratic interpolation with the parameter $1/2$ between the spaces $H^{s-1/2,\varphi,(2q)}(\Omega)$ and $H^{s+1/2,\varphi,(2q)}(\Omega)$.

If $s>2q-1/2$, the spaces $H^{s,\varphi,(2q)}(\Omega)$ and $H^{s,\varphi}(\Omega)$ are equal as completions of $C^{\infty}(\overline{\Omega})$ with respect to equivalent norms. In the opposite case, the space $H^{s,\varphi,(2q)}(\Omega)$ contains elements that are not distributions. If $s_1<s_2$ and $\varphi_1,\varphi_2\in\mathcal{M}$, the identity mapping on $C^{\infty}(\overline{\Omega})$ extends uniquely to a compact embedding of $H^{s_2,\varphi_2,(2q)}(\Omega)$ in $H^{s_1,\varphi_1,(2q)}(\Omega)$.

\begin{proof}[Proof of Theorem~$\ref{10th1}$]
We arbitrarily choose a function $\chi\in C^{\infty}(\overline{\Omega})$ subject to $\mathrm{supp}\,\chi\subset\Omega_0\cup\nobreak\Gamma_0$ and consider a function $\eta\in C^{\infty}(\overline{\Omega})$ such that $\mathrm{supp}\,\eta\subset\Omega_0\cup\Gamma_0$ and that $\eta=1$ in some neighbourhood $V$ of $\mathrm{supp}\,\chi$ in the topology of $\overline{\Omega}$. According to the hypotheses of the theorem, we have the inclusion $(u,v)\in\mathcal{D}^{l}_{A}(\Omega,\Gamma)$ for a sufficiently small integer $l<s$ and the inclusion $\eta(f,g)\in\mathcal{E}^{0,s,\varphi}(\Omega,\Gamma)$. We must deduce from them that $\chi(u,v)\in\mathcal{D}^{s,\varphi}(\Omega,\Gamma)$.

Since the operator \eqref{10f5} is Fredholm and since the set
$$
\mathcal{E}^{\infty}(\overline{\Omega},\Gamma):=
C^{\infty}(\overline{\Omega})\times(C^{\infty}(\Gamma))^{q+\varkappa}
$$
is dense in $\mathcal{E}^{0,s,\varphi}(\Omega,\Gamma)$, it follows from
\cite[Lemma~2.1]{HohbergKrein60} that $\eta(f,g)=\Lambda(u',v')+(f'',g'')$ for some $(u',v')\in\mathcal{D}_{A}^{s,\varphi}(\Omega,\Gamma)$ and $(f'',g'')\in\mathcal{E}^{\infty}(\overline{\Omega},\Gamma)$. Then
\begin{equation}\label{1proof-f9}
\Lambda(u-u',v-v')=(1-\eta)(f,g)+(f'',g'')=:(f''',g'''),
\end{equation}
with
\begin{equation}\label{1proof-f10}
\zeta(f''',g''')=
\zeta(f'',g'')\in\mathcal{E}^{\infty}(\overline{\Omega},\Gamma)=
\bigcap_{\sigma\in\mathbb{R}}\mathcal{E}_{\sigma,(0)}(\Omega,\Gamma)
\end{equation}
for every function $\zeta\in C^{\infty}(\overline{\Omega})$ subject to $\mathrm{supp}\,\zeta\subset V$. Here and below,
$$
\mathcal{E}_{\sigma,\varphi,(0)}(\Omega,\Gamma):=
H^{\sigma-2q,\varphi,(0)}(\Omega)\oplus
\bigoplus_{j=1}^{q+\varkappa}H^{\sigma-m_{j}-1/2,\varphi}(\Gamma)
$$
and
$$
\mathcal{D}^{\sigma,\varphi,(2q)}(\Omega,\Gamma):= H^{\sigma,\varphi,(2q)}(\Omega)\oplus
\bigoplus_{k=1}^{\varkappa}H^{\sigma+r_{k}-1/2,\varphi}(\Gamma)
$$
for every $\sigma\in\mathbb{R}$.

It follows from \cite[Section~4.4.2, isomorphism (4.196)]{MikhailetsMurach14} that the identity mapping on $C^{\infty}(\overline{\Omega})$ extends uniquely to a continuous embedding $H^{l}_{A}(\Omega)\hookrightarrow H^{l,(2q)}(\Omega)$. Hence, we may consider the distribution $u-u'\in H^{l}_{A}(\Omega)$ as an element of the Roitberg space $H^{l,(2q)}(\Omega)$. Thus,
\begin{equation}\label{1proof-f11}
(u-u',v-v')\in\mathcal{D}^{l,(2q)}(\Omega,\Gamma).
\end{equation}

According to the lifting property \cite[Theorem~2.4.3]{Roitberg99} (see also \cite[Theorem~4]{ChepurukhinaMurach15MFAT1} as to the spaces $H^{\sigma,\varphi,(2q)}(\Omega)$), it follows from \eqref{1proof-f9}--\eqref{1proof-f11} that
$$
(u^{\circ},v^{\circ}):=\chi(u-u',v-v')\in \bigcap_{\sigma\in\mathbb{R}}\mathcal{D}^{\sigma,(2q)}(\Omega,\Gamma)
\subset\mathcal{D}^{s,\varphi}(\Omega,\Gamma)
$$
because $\mathrm{supp}\,\chi\subset V$. Hence,
$$
\chi(u,v)=(u^{\circ},v^{\circ})+\chi(u',v')
\in\mathcal{D}^{s,\varphi}(\Omega,\Gamma)
$$
in view of the inclusion $(u',v')\in\mathcal{D}_{A}^{s,\varphi}(\Omega,\Gamma)$.
\end{proof}

\begin{proof}[Proof of Theorem~$\ref{10th2}$]
According to \cite[Theorem~1]{ChepurukhinaMurach15MFAT1}, the mapping \eqref{mapping} extends uniquely (by continuity) to a Fredholm bounded operator
\begin{equation}\label{proof2-f12}
\Lambda:\mathcal{D}^{\sigma,\varphi,(2q)}(\Omega,\Gamma)\to
\mathcal{E}_{\sigma,\varphi,(0)}(\Omega,\Gamma)
\quad\mbox{for every}\;\;\sigma\in\mathbb{R},
\end{equation}
the kernel and index of this operator being the same as those of the operator~\eqref{Fredholm-positive-scale}. Note that these operators coincide if $s=\sigma>2q-1/2$. Let $\|\cdot\|'_{\sigma,\varphi,(2q)}$ denote the norm in $\mathcal{D}^{\sigma,\varphi,(2q)}(\Omega,\Gamma)$, and let $\|\cdot\|''_{\sigma,\varphi,(0)}$ denote the norm in $\mathcal{E}_{\sigma,\varphi,(0)}(\Omega,\Gamma)$.

Assume that $0\leq l\in\mathbb{Z}$ and that a function $\zeta\in C^{\infty}(\overline{\Omega})$ satisfies the condition $\zeta=1$ in a neighbourhood of $\mathrm{supp}\,\chi$. Let us prove by induction in $l$ that
\begin{equation}\label{proof2-f13}
\|\chi(u,v)\|'_{s,\varphi,(2q)}\leq c_{0}\bigl(\|\zeta\Lambda(u,v)\|''_{s,\varphi,(0)}+
\|\zeta(u,v)\|'_{s-l,\varphi,(2q)}\bigr)
\end{equation}
for every $(u,v)\in\mathcal{D}^{\infty}(\overline{\Omega},\Gamma)$ with some number $c_{0}>0$ not depending on $(u,v)$.

If $l=0$, then \eqref{proof2-f13} follows from the evident fact that the operator of the multiplication by a function from $C^{\infty}(\overline{\Omega})$ is bounded on every space $H^{\sigma,\varphi,(2q)}(\Omega)$. Assume now that the inequality \eqref{proof2-f13} holds true for a certain integer $l=p\geq0$, and prove this inequality for $l=p+1$.

Consider a function $\zeta_{0}\in C^{\infty}(\overline{\Omega})$ such that  $\zeta_{0}=1$ in a neighbourhood of $\mathrm{supp}\,\chi$ and that $\zeta=1$ in a neighbourhood of $\mathrm{supp}\,\zeta_{0}$. By the inductive assumption,
\begin{equation}\label{proof2-f14}
\|\chi(u,v)\|'_{s,\varphi,(2q)}\leq c_{1}\bigl(\|\zeta_{0}\Lambda(u,v)\|''_{s,\varphi,(0)}+
\|\zeta_{0}(u,v)\|'_{s-p,\varphi,(2q)}\bigr).
\end{equation}
In the proof, $c_{1}$, $c_{2}$,... denote some positive numbers that do not depend on $(u,v)$. Since the bounded operator \eqref{proof2-f12}, where $\sigma=s-p$, is Fredholm, we conclude by Peetre's lemma \cite[Lemma~3]{Peetre61} that
\begin{equation}\label{proof2-f15}
\|\zeta_{0}(u,v)\|'_{s-p,\varphi,(2q)}\leq c_{2}\bigl(\|\Lambda(\zeta_{0}(u,v))\|''_{s-p,\varphi,(0)}+
\|\zeta_{0}(u,v)\|'_{s-p-1,\varphi,(2q)}\bigr).
\end{equation}
Interchanging the operator of the multiplication by $\zeta_{0}$ with the PDOs used in the problem \eqref{10f1}, \eqref{10f2}, we write
\begin{equation}\label{proof2-f16}
\begin{aligned}
\Lambda(\zeta_{0}(u,v))&=\Lambda(\zeta_{0}\zeta(u,v))=
\zeta_{0}\Lambda(\zeta(u,v))+\Lambda'(\zeta(u,v))\\
&=\zeta_{0}\Lambda(u,v)+\Lambda'(\zeta(u,v)),
\end{aligned}
\end{equation}
where $\Lambda'$ is an operator of the same structure as $\Lambda$ but formed by PDOs of lower orders than the corresponding PDOs in \eqref{10f1}, \eqref{10f2}. Hence,
\begin{equation}\label{proof2-f17}
\|\Lambda'(\zeta(u,v))\|''_{s-p,\varphi,(0)}\leq
c_{3}\|\zeta(u,v)\|'_{s-p-1,\varphi,(2q)}
\end{equation}
due to \cite[Theorem~4.13]{MikhailetsMurach14}.
According to \eqref{proof2-f15}--\eqref{proof2-f17}, we obtain the inequality
\begin{align*}
\|\zeta_{0}(u,v)\|'_{s-p,\varphi,(2q)}\leq c_{4}\bigl(\|\zeta_{0}\Lambda(u,v)\|''_{s-p,\varphi,(0)}+
\|\zeta(u,v)\|'_{s-p-1,\varphi,(2q)}\bigr).
\end{align*}
Substituting it in \eqref{proof2-f14}, we arrive at \eqref{proof2-f13} in the $l=p+1$ case. Thus, \eqref{proof2-f13} is proved for every integer $l\geq0$.

Choose a number $p>\lambda$ such that $s-p$ is a negative integer. It follows from the inequality \eqref{proof2-f13} for an integer $l>p$ that
\begin{equation}\label{proof2-f18}
\|\chi(u,v)\|'_{s,\varphi,(2q)}\leq c_{5}\bigl(\|\zeta\Lambda(u,v)\|''_{s,\varphi,(0)}+
\|\zeta(u,v)\|'_{s-p,(2q)}\bigr)
\end{equation}
if we take \eqref{10f3} into account. Let us deduce the required estimate \eqref{10f9} from \eqref{proof2-f18}. We continue to assume that $(u,v)\in \mathcal{D}^{\infty}(\overline{\Omega},\Gamma)$.

By the definition of $H^{s,\varphi,(2q)}(\Omega)$, we have
\begin{equation*}
\|\chi u\|_{s,\varphi;\Omega}\leq \|\chi u\|_{s,\varphi,(0);\Omega}\leq
\|\chi u\|_{s,\varphi,(2q);\Omega}
\end{equation*}
if $s\notin\{1/2,3/2,\ldots,2q-1/2\}$. It follows from this by the quadratic interpolation that
\begin{equation*}
\|\chi u\|_{s,\varphi;\Omega}\leq c_{6}\|\chi u\|_{s,\varphi,(2q);\Omega}
\end{equation*}
for the rest values of $s$. Hence,
\begin{equation}\label{proof2-f19}
\|\chi(u,v)\|'_{s,\varphi}\leq c_{7}\|\chi(u,v)\|'_{s,\varphi,(2q)}.
\end{equation}

Let $W$ be an open set from the topology on $\overline{\Omega}$ such that $\mathrm{supp}\,\chi\subset W$ and that $\eta=1$ on $\overline{W}$ and that $W_0:=W\cap\Omega$ is an open domain in $\mathbb{R}^{n}$ with infinitely smooth boundary. The last condition allows us to consider the
Roitberg space $H^{s-p,(2q)}(W_0)$. Let $w\in C^{\infty}(\overline{W})$ be the restriction of $u$ to $\overline{W}$. Assume in addition that  $\mathrm{supp}\,\zeta\subset W$. We have the equivalence of norms
\begin{equation}\label{proof2-f20}
\|\zeta u\|_{s-p,(2q),\Omega}\asymp\|\zeta w\|_{s-p,(2q);W_0}.
\end{equation}
Indeed, owing to \cite[Theorem 6.1.1]{Roitberg96} and since $s-p<0$, we get
\begin{align*}
\|\zeta u\|_{s-p,(2q);\Omega}&\asymp
\|\zeta u\|_{s-p,(0);\Omega}+\|A(\zeta u)\|_{s-p-2q,(0);\Omega}\\
&=\|\mathcal{O}(\zeta u)\|_{s-p;\mathbb{R}^{n}}+
\|\mathcal{O}A(\zeta u)\|_{s-p-2q;\mathbb{R}^{n}}\\
&=\|\zeta w\|_{s-p,(0);W_0}+\|A(\zeta w)\|_{s-p-2q,(0);W_0}\asymp
\|\zeta w\|_{s-p,(2q);W_0}
\end{align*}
because $\mathcal{O}(\zeta u)$ and $\mathcal{O}A(\zeta u)$ are also extensions of the functions $\zeta w$ and $A(\zeta w)$, resp., to $\mathbb{R}^{n}$ with zero. According to \cite[Section~4.4.2, isomorphism (4.196)]{MikhailetsMurach14} we have another equivalence of norms
\begin{equation}\label{proof2-f21}
\|w\|_{s-p,(2q),W_0}+\|Aw\|_{W_0}\asymp
\|w\|_{s-p,W_0}+\|Aw\|_{W_0};
\end{equation}
here, recall, $\|\cdot\|_{W_0}$ denotes the norm in $L_{2}(W_0)$.

Formulas \eqref{proof2-f20} and \eqref{proof2-f21} yield
\begin{equation}\label{proof2-f22}
\begin{aligned}
\|\zeta u\|_{s-p,(2q);\Omega}&\asymp\|\zeta w\|_{s-p,(2q);W_0}\leq
c_{8}\bigl(\|w\|_{s-p,(2q);W_0}+\|Aw\|_{W_0}\bigr)\\
&\asymp\|w\|_{s-p;W_0}+\|Aw\|_{W_0}\leq
\|\eta u\|_{s-p;\Omega}+\|\eta Au\|_{\Omega}.
\end{aligned}
\end{equation}
Substituting \eqref{proof2-f19} and \eqref{proof2-f22} in \eqref{proof2-f18}, we get
\begin{align*}
\|\chi(u,v)\|'_{s,\varphi}&\leq c_{9} \bigl(\|\zeta\Lambda(u,v)\|''_{s,\varphi,(0)}+
\|\eta(u,v)\|'_{s-p}+\|\eta Au\|_{\Omega}\bigr)\\
&\leq c\bigl(\|\eta\Lambda(u,v)\|''_{0,s,\varphi}+
\|\eta(u,v)\|'_{s-\lambda,\varphi}\bigr)
\end{align*}
because $s-2q<0$ (then the norm $\|\cdot\|_{s-2q,(0);\Omega}$ is subordinate to $\|\cdot\|_{\Omega}$) and because $p>\lambda$ (then
the norm $\|\cdot\|'_{s-p}$ is subordinate to $\|\cdot\|'_{s-\lambda,\varphi}$). Thus, we have proved the required estimate \eqref{10f9} in the case where $(u,v)\in \mathcal{D}^{\infty}(\overline{\Omega},\Gamma)$.

Now we consider an arbitrary vector \eqref{10f6} that satisfies the hypotheses of Theorem~$\ref{10th1}$ and deduce this estimate from the case just examined. Let $V$ be an open set from the topology on $\overline{\Omega}$ such that $\overline{V}\subset\Omega_{0}\cup\Gamma_{0}$ and $\mathrm{supp}\,\eta\subset V$ and that $V_0:=V\cap\Omega$ is an open domain in $\mathbb{R}^{n}$ with an infinitely smooth boundary $\partial V_0$. According to Theorem~$\ref{10th1}$, we have the inclusion $\omega:=u\!\upharpoonright\!V_0\in H^{s,\varphi}_{A}(V_0)$. Since  $C^{\infty}(\overline{V})$ is dense in $H^{s,\varphi}_{A}(V_0)$, there exists a sequence $(u_{r})_{r=1}^{\infty}\subset C^{\infty}(\overline{\Omega})$ such that $\omega_{r}:=u_{r}\!\upharpoonright\!\overline{V}\to\omega$ in $H^{s,\varphi}(V_0)$ and $A\omega_{r}\to A\omega$ in $L_{2}(V_0)$ as $r\to\infty$. Then
\begin{equation}\label{proof2-f23}
\eta u_{r}\to\eta u\quad\mbox{in}\;\;
H^{s,\varphi}(\Omega)
\end{equation}
and
\begin{equation}\label{proof2-f24}
\eta Au_{r}\to\eta Au\quad\mbox{in}\;\;L_{2}(\Omega)
\end{equation}
as $r\to\infty$. The second convergence is evident; let us explain the first. Since $\omega_{r}-\omega\to0$ in $H^{s,\varphi}(V_0)$, there exists a sequence $(\omega_{r}^{\circ})_{r=1}^{\infty}\subset H^{s,\varphi}(\mathbb{R}^{n})$ such that $\omega_{r}^{\circ}=\omega_{r}-\omega$ in $V_0$ and that  $\omega_{r}^{\circ}\to0$ in $H^{s,\varphi}(\mathbb{R}^{n})$. Then $\eta(u_{r}-u)=\eta\cdot \omega_{r}^{\circ}\!\upharpoonright\!\Omega\to0$ in $H^{s,\varphi}(\Omega)$, which gives~\eqref{proof2-f23}.

Let us deduce from the convergence $\omega_{r}\to\omega$ in $H^{s,\varphi}_{A}(V_0)$ that
\begin{equation}\label{proof2-f25}
\eta B_{j}u_{r}\to\eta B_{j}u\quad\mbox{in}\;\;
H^{s-m_j-1/2,\varphi}(\Gamma)
\end{equation}
as $r\to\infty$ for every $j\in\{1,\ldots,q+\varkappa\}$. Given such $j$, we consider a boundary PDO on $\partial V_0$ of the form
\begin{equation*}
B_{j}^{\star}:=B_{j}^{\star}(x,D):=
\sum_{|\mu|\leq m_j}b_{j,\mu}^{\star}(x)D^{\mu}
\end{equation*}
where each coefficient $b_{j,\mu}^{\star}$ belongs to $C^{\infty}(\partial V_0)$ and coincides with the corresponding coefficient $b_{j,\mu}$ of $B_{j}$ on $\Gamma\cap\partial V_0$. Then
\begin{equation*}
B_{j}^{\star}\omega_{r}\to B_{j}^{\star}\omega\quad\mbox{in}\quad H^{s-m_{j}-1/2,\varphi}(\partial V_0)
\end{equation*}
due to Proposition~\ref{10pr1} considered for $V_0$ instead of $\Omega$ (as is seen from \cite[Proof of Theorem~1]{MurachChepurukhina15UMJ5}, the boundedness of the operator \eqref{10f5} does not depend on property (ii) of boundary conditions given in Section~\ref{sec2}). Since $\eta B_{j}^{\star}\omega_{r}=\eta B_{j}u_{r}$ on $\Gamma\cap\partial V_0$, we get
\begin{equation}\label{proof2-f26}
\eta B_{j}u_{r}\to T(\eta B_{j}^{\star}\omega)\quad\mbox{in}\;\;
H^{s-m_j-1/2,\varphi}(\Gamma),
\end{equation}
where the distribution $T(\eta B_{j}^{\star}\omega)$ is equal by definition to $\eta B_{j}^{\star}\omega$ on $\Gamma\cap V$ and to zero on $\Gamma\setminus\mathrm{supp}\,\eta$.

Note that
\begin{equation}\label{proof2-f27}
\eta B_{j}^{\star}\omega=\eta B_{j}u\quad\mbox{on}\;\;\Gamma\cap V.
\end{equation}
Indeed, since $u\in H^{l}_{A}(\Omega)$ for some $l<s$, there exists a sequence $(u_{r}^{\ast})_{r=1}^{\infty}\subset C^{\infty}(\overline{\Omega})$ that converges to $u$ in $H^{l}_{A}(\Omega)$. Hence, $\eta B_{j}u_{r}^{\ast}\to\eta B_{j}u$ in $H^{l-m_j-1/2}(\Gamma)$ due to Proposition~\ref{10pr1}. Besides, since  $u_{r}^{\circ}:=u_{r}^{\ast}\!\upharpoonright\!V_0\to u\!\upharpoonright\!V_0=\omega$ in $H^{l}_{A}(V_0)$, we have the convergence $\eta B_{j}^{\star}u_{r}^{\circ}\to\eta B_{j}^{\star}\omega$ in $H^{l-m_j-1/2}(\partial V_0)$. However, $\eta B_{j}u_{r}^{\ast}=\eta B_{j}^{\star}u_{r}^{\circ}$ on $\Gamma\cap\partial V_0\supset\Gamma\cap V$. The last two limits therefore yield property \eqref{proof2-f27}. Owing to it, we have the equality $T(\eta B_{j}^{\star}\omega)=\eta B_{j}u$ on $\Gamma$, which together with \eqref{proof2-f26} gives \eqref{proof2-f25}.

Consider a function $\eta_{1}\in C^{\infty}(\Gamma)$ such that $\mathrm{supp}\,\eta_{1}\subset\Gamma_{0}$ and that $\eta_{1}=1$ in a neighbourhood of $\Gamma\cap\mathrm{supp}\,\eta$ (in the topology on $\Gamma$, of course). According to Theorem~$\ref{10th1}$, the inclusion $\eta_{1}v_{k}\in H^{s+r_k-1/2,\varphi}(\Gamma)$ holds true for each $k\in\{1,\ldots,\varkappa\}$. We choose a sequence $(v^{(r)}_{k})_{r=1}^{\infty}\subset C^{\infty}(\Gamma)$ such that
\begin{equation}\label{proof2-f28}
v^{(r)}_{k}\to\eta_{1}v_{k}\quad\mbox{in}\;\;H^{s+r_k-1/2,\varphi}(\Gamma)
\end{equation}
as $r\to\infty$. Then
\begin{equation}\label{proof2-f29}
\eta C_{j,k}v^{(r)}_{k}\to\eta C_{j,k}v_{k}\quad \mbox{in}\;\;H^{s-m_j-1/2,\varphi}(\Gamma)
\end{equation}
for all admissible values $j$ and $k$; see \cite[Lemma~2.5]{MikhailetsMurach14}. Put $v^{(r)}:=(v^{(r)}_{1},\ldots,v^{(r)}_{\varkappa})$. As we have proved, the inequality \eqref{10f9} holds true for $(u_{r},v^{(r)})\in\mathcal{D}(\overline{\Omega},\Gamma)$ instead of $(u,v)$, i.e.
\begin{equation*}
\|\chi(u_{r},v^{(r)})\|'_{s,\varphi}\leq c\,\bigl(\|\eta\Lambda(u_{r},v^{(r)})\|''_{0,s,\varphi}+
\|\eta(u_{r},v^{(r)})\|'_{s-\lambda,\varphi}\bigr).
\end{equation*}
Passing here to the limit as $r\to\infty$ and using \eqref{proof2-f23}--\eqref{proof2-f25}, \eqref{proof2-f28}, and \eqref{proof2-f29}, we obtain the required estimate \eqref{10f9}.
\end{proof}

\bigskip

\emph{Acknowledgement.} The authors have received funding from the European Union's Horizon 2020 research and innovation programme under the Marie Sk{\l}odowska-Curie grant agreement No 873071.

\medskip

{\small

\end{document}